\documentclass[12pt,leqno]{article}
\usepackage{amsfonts, amsmath, amssymb, amsthm}  
\usepackage[english]{babel}
\usepackage[utf8]{inputenc}
\usepackage[OT1]{fontenc}
\usepackage{bm}

\usepackage{textcomp, cmap, comment}	
\usepackage{graphicx, wrapfig}
\usepackage{epigraph, xcolor, hyperref}
\usepackage{array,tabularx,tabulary,booktabs}
\usepackage{euscript, mathrsfs}	
\usepackage[a4paper]{geometry}

\newtheorem*{thm*}{Theorem}
\newcommand{\ff}{{\mathcal F}}

\newtheorem*{cla*}{Claim}

\newcommand{\G}{{\mathcal G}}

\newtheorem{thm}{Theorem}

\newtheorem{prb}{Problem}

\date{}

\title{Choice number of Kneser graphs}
\author{Vera Bulankina\footnote{Moscow Institute of Physics and Technology, Russia} \, and Andrey Kupavskii\footnote{G-SCOP, CNRS, University Grenoble-Alpes, France and
Moscow Institute of Physics and Technology, Russia; Email: {\tt kupavskii@ya.ru} }}

\date{}

\begin{document}
\maketitle
\begin{abstract}
In this short note, we show that for any $\epsilon >0$ and $k<n^{0.5-\epsilon}$ the choice number of the Kneser Graph $KG_{n,k}$ is $\Theta (n\log n)$.\end{abstract}

\section{Introduction}
Let $[n] = \{1,\ldots, n\}$ be the standard $n$-element set and, for a set $X$, let ${X\choose k}$ stand for all $k$-element subsets of $X$ ($k$-sets for short). For integers $n\ge 2k>0$, the Kneser graph $KG_{n,k}$ is a graph with the vertex set ${[n]\choose k}$ and  the edge set that consists of all pairs of disjoint $k$-sets.  

Recall that, for a graph $G$, the quantity $\chi(G)$ is the smallest number $s$ of colors such that there is a vertex coloring in $s$ colors in which the endpoints of each edge receive different colors (a {\it proper coloring}). The choice number $ch(G)$ is the smallest $s$ such that for any assignment of lists $S(v)$ of size $s$ to each vertex $v\in G$  there is a proper coloring of the vertices of $G$ that uses the color from $S(v)$ for each $v$.

It is by now one of the classical results in combinatorics that $\chi(KG_{n,k}) = n-2k+2$. It was shown by Lov\'asz \cite{L}, answering the question by Kneser. In fact, the upper bound is easy: for each $1\le i\le n-2k+1,$ color in $i$ the sets with minimum element $i$. The remaining $k$-sets are subsets of $\{n-2k+2,\ldots, n\}$ and do not induce an edge in $KG_{n,k}.$ Thus, they can be colored in one color.

Lov\'asz' paper initiated the use of topological method in combinatorics. By now, different proofs \cite{Bar, Mat} of Lov\'asz's result are known; however, all of them rely on topological arguments. Rather quickly after B\'ar\'any's proof, Schrijver \cite{Sch} constructed vertex-critical subgraphs of Kneser graphs, that is, subgraphs with the same chromatic number such that the deletion of any vertex decreases the chromatic number. These subgraphs are induced subgraphs of $KG_{n,k}$ on the vertices that correspond to $k$-sets that do not contain two cyclically consecutive elements. Very recently, Kaiser and Stehl\'ik \cite{KS} constructed edge-critical subgraphs of Schrijver graphs. There, deletion of any edge decreases chromatic number. After a series of papers \cite{Kup1, AH, Kup2}, the second author and Kiselev \cite{KK} essentially determined the chromatic number of a random subgraph of $KG_{n,k},$ obtained by including each edge with probability $1/2.$ There are extensions of Lov\'asz' and Schrijver's results to hypergraphs \cite{AFL}.

In this note, we study the choice number of Kneser graphs. In what follows, $\log x$ stands for the natural logarithm of $x.$

\begin{thm}\label{thmub}
For any $n\ge 2k>0$ we have $ch(KG_{n,k}) \leq n \log\frac{n}{k} + n.$
\end{thm}
It should be clear that $ch(KG_{n,k})\ge \chi(KG_{n,k})=n-2k+2$, and thus Theorem~\ref{thmub} implies that $ch(KG_{n,k}) = \Theta (n)$ for $C_1n\le k \le C_2n$, where $0<C_1\le C_2<\frac 12$.
The following result improves on this lower bound  for relatively small $k$.

\begin{thm}\label{thm4}
Fix $s\ge 3.$ If $n$ is sufficiently large and $3\le k\le n^{\frac 12-\frac 1s}$ then $ch(KG_{n,k})\ge \frac 1{2s^2} n\log n$. For $k=2$ we have  $ch(KG_{n,k})\ge \frac 1{32} n\log n$ for sufficiently large $n$.
\end{thm}

For $k=2$ we can improve the bound to $ch(KG_{n,k})\ge \frac 1{4} n\log n$, but it requires a different proof which we decided to omit.
These two results leave open the following intriguing question.
\begin{prb}
Determine the asymptotics of $ch(KG_{n,k})$ for $ \Omega(\sqrt n)= k = o(n)$.
\end{prb}

\section{Proofs}
\begin{proof}[Proof of Theorem~\ref{thmub}]
We shall employ the probabilistic method. Let $S(v)$ be the list of $m$ colors assigned to vertex $v$. Denote by $L$ the set of all colors assigned to at least $1$ vertex. In what follows, we slightly abuse notation and identify vertices of $KG_{n,k}$ and the corresponding $k$-sets.

Let us take a random map $f: L\to [n]$. Such a correspondence induces a coloring of $KG_{n,k}$ as follows. 
We color a $k$-set $v$ in color $\gamma$, $\gamma\in S(v)$, if there is an element $i\in [n]$ such that, first, $i\in v$ and, second, $f(\gamma) = i.$ (If there are several such $\gamma$, then we use any of them.) It should be clear that such coloring, if it exists, is proper and respects the color lists. Indeed, if two sets share the same color $\gamma$, then they must share a common element $f(\gamma).$

The last part of the proof is to show that the probability that such a coloring exists is non-zero. 
The probability that a vertex $v$ is not colored is at most  $(1-\frac{k}{n})^{m}<e^{-m\frac{k}{n}}.$ Then, the probability that there is at least one vertex that is not colored is at most  ${ n \choose k} \cdot e^{-m\frac{k}{n}}$. If this probability is strictly smaller than $1$, then with positive probability the opposite holds, and we have a proper coloring. Let us bound this last expression.
$${n \choose k} \cdot e^{-m\frac{k}{n}} < \left(\frac{ne}{k}\right)^k \cdot e^{-m\frac{k}{n}} = e^{k(\log(\frac{ne}{k})-\frac{m}{n})} < 1, $$ if $m > n \log\frac{n}{k} + n.$
\end{proof}

\subsection{Proof of Theorem~\ref{thm4}}
We shall need the following structural result concerning intersecting families. Recall that a family $\ff\subset 2^{[n]}$ is {\it intersecting } if $F_1\cap F_2\ne \emptyset $ for any $F_1,F_2\in \ff.$ 
For a family $\ff$ and a set $S$, let us use the following notation: $$\ff(S):=\{F\in \ff: S\subset F\}.$$
\begin{thm}\label{thmintstruc} Consider an intersecting family $\ff \subset {[n]\choose k}$ and fix an integer $s\ge 2.$ Then either there exists a family $\G\subset {[n]\choose s}$, $|\G|\le k^s,$ such that $$\ff\subset \bigcup_{S\in \G}\ff(S),$$ 
or a set $I\subset [n]$ of size at most $s-1$ such that $$\ff\subset \bigcup_{i\in I}\ff(\{i\}).$$
\end{thm}
\begin{proof}
 Recall that a {\it cover} of the family $\ff$ is a set $C$  with $F\cap C\ne \emptyset$ for any $F\in \ff.$ The {\it covering number} $\tau(\ff)$  is the minimum size of a cover of $\ff$.

If $\tau(\ff)\le s-1$ then simply take $I$ to be the smallest cover of $\ff.$

If $\tau(\ff)\ge s$ then we shall construct $\G$ using the following simple inductive argument for an intersecting $\ff,$ which is inspired by the paper of Erd\H os and Lov\' asz \cite{EL} (cf. also \cite{KK2}). 

Take an arbitrary set $F\in \ff.$  Define $\G_1\subset {[n]\choose 1}$ as follows: $\G_1:=\{\{i\}: i\in F\}.$ Then $\ff\subset \cup_{i\in F}\ff(i)$ since $\ff$ is intersecting. 


For each  $1\le \ell<s$, let us show how to construct $\G_{\ell+1}$ from $\G_\ell$. Assume that we have a family $\G_\ell\subset {[n]\choose \ell}$ of at most $k^\ell$ sets such that $\ff\subset \cup_{G\in \G_{\ell}}\ff(G)$. For each set $G\in \G_\ell$, consider a set $F_G\in\ff$ that is disjoint with $G$. Such a set must exist since $|G|<\tau(\ff)$. Put $\G_{\ell+1}:=\{G\cup \{i\}: G\in \G_{\ell}, i\in F_G\}$. It should be clear that $|\G_{\ell+1}|\le k^{\ell+1}$ and that $$\ff\subset \bigcup_{G'\in \G_{\ell+1}}\ff(G').$$ 
Finally, we put $\G:=\G_s$.
\end{proof}

We shall also need the Tur\'an-type result for hypergraphs due to Katona, Nemeth and Simonovits \cite{KNS}. Recall that, for a hypergraph $\mathcal H$, its {\it independence number} $\alpha(\mathcal H)$ is the size of the largest subset of vertices that does not contain any edge of $\mathcal H$.
\begin{thm}[\cite{KNS}]\label{thmtur}
If $\mathcal H\subset {X\choose s}$ is a hypergraph with $\alpha(\mathcal H) = q$ then $$|\mathcal H|\ge {|X|\choose s}/{q\choose s}\ge \Big(\frac {|X|}{q}\Big)^s.$$
\end{thm}

\begin{proof}[Proof of Theorem~\ref{thm4}] We again employ the probabilistic method. For shorthand, put $u: = \frac 1{s^2}n\log n$.\footnote{We tacitly assume that $u$ is an integer.} Take a set of $u$ colors and correspond to each vertex of $KG_{n,k}$ a random subset of colors of size $u/2.$ 

Take an arbitrary independent set in $KG_{n,k}$ (i.e., an intersecting family in ${[n]\choose k}$) and fix an integer $s\ge 2.$ Using Theorem~\ref{thmintstruc}, we get that each such independent set is contained in one of the families from $\mathcal C$, where $\mathcal C$ consists of all families $\mathcal K$ of the following two forms:\footnote{Note that $\mathcal C$ is a family of families.} 
\begin{itemize}
    \item[type A:] all $k$-sets that intersect a fixed set $I(\mathcal K),$   $|I| =s-1$;
    \item[type B:] all $k$-sets that contain one of the $s$-sets from a family $\G(\mathcal K)\subset {[n]\choose s},$ $|\G(\mathcal K)|=k^s.$
\end{itemize}
Note that $|\mathcal C|\le n^{s-1}+{n\choose s}^{k^s}$.  We say that a coloring $X =X_1\sqcup\ldots \sqcup X_m$ of a set $X$ {\it lies} in a cover $X =X_1'\cup\ldots \cup X_m'$ of the same set if $X_i\subset X_i'$ for each $i.$ Using this terminology, any possible partition of ${[n]\choose k}$ into $u$ independent sets lies in one of the 
\begin{equation}\label{eqnumbercolor}|\mathcal C|^u \le \Big({n\choose s}^{k^s}+n^{s-1}\Big)^{u}\le n^{sk^su}\end{equation}
covers, formed by a $u$-tuple of families from $\mathcal C$. 

 For a given cover $\mathcal K = \mathcal K_1\cup\ldots\cup \mathcal K_u$ from $\mathcal C,$ let us bound from above the probability of the event $A_{\mathcal K}$ that $KG_{n,k}$ can be colored in one of the colorings of ${[n]\choose k}$ that lie in $\mathcal K$ and that respects the lists assigned to the $k$-sets. For each $\ell\in [n]$, define $d_\ell = |\{j\in [u]: \mathcal K_j\text{ is of type A and }\ell\in I(\mathcal K_j)\}$. Then, clearly, $\sum_{i\in[n]} d_i \le (s-1)u$, and, thus, we get that for any $0<\epsilon<1$ there is a set $W\subset [n]$ of $\epsilon n$ elements such that \begin{equation}\label{eqlowdeg}d_i\le (1-\epsilon)^{-1}(s-1)u/n\text{ for any }i\in W.\end{equation}
 
 (Otherwise, $\sum_{i\in [n]\setminus W}d_i>(n-|W|)(1-\epsilon)^{-1}(s-1)u/n>(s-1)u.$)
 Next, define an $s$-uniform hypergraph 
 $$\mathcal H:=\big\{H\subset W: \mathcal K_j \text{ is of type B and }H\in \G(\mathcal K_j)\text{ for some }j\big\}.$$
 Note that $|\mathcal H|\le k^s u.$ Applying Theorem~\ref{thmtur}, we get that 
 $$\alpha(\mathcal H)\ge \frac{|W|}{|\mathcal H|^{1/s}}\ge \frac{\epsilon n}{ku^{1/s}}=:t.$$
 Thus, there exists a set $Y\subset W$ of size $t$ that is independent in $\mathcal H.$ 
 
 Take any $X\in {Y\choose k}$ and denote by $B_{X,\mathcal K}$ the event that $X$ cannot be colored via a coloring that lies in $\mathcal K$. As a vertex of $KG_{n,k},$ $X$ cannot be colored in color $j$ if $\mathcal K_j$ is of type B, because $X$ is independent in $\mathcal H$. Next, it may be colored in $j$ with $\mathcal K_j$ of type A only if $j$ belongs to the randomly chosen subset of $u/2$ colors that was assigned to $X$. Recall that \eqref{eqlowdeg} holds. Denote $z:=(1-\epsilon)^{-1}k(s-1)u/n$ and note that $z=o(\sqrt n)$. Then the probability $p$ that $X$ cannot be colored using this coloring is at least the probability that the colors for $X$ were chosen from the complement of the set $\cup_{\ell\in X}\{j\in [u]: \mathcal K_j\text{ is of type A and }\ell\in I(\mathcal K_j)\}$. This set has size $\sum_{\ell\in X}d_i\le z$, and so we get that 
 $$p\ge \frac{{u-\sum_{i\in X}d_i\choose u/2}}{{u\choose u/2}} \ge \frac{{u-z\choose u/2}}{{u\choose u/2}}=\prod_{j=0}^{z-1}\frac{\frac u2-j}{u-j}\ge 2^{-z}\Big(1-\sum_{j=0}^{z-1}\frac{j}{u-j}\Big)=(1+o(1))2^{-z}.$$
 Expanding the expressions for $z$ and $u$, we get that $$p = (1+o(1))2^{-z} = (1+o(1))n^{-(1-\epsilon)^{-1}k\frac {s-1}{s^2}}\ge n^{-\frac k{s+1}},$$ provided $\epsilon>0$ is sufficiently small.
 
 Thus, $A_{\mathcal K} \subset \bigcap_{X\in {Y\choose k}}\bar B_{X,\mathcal K}$ and so 
 $$\Pr[A_{\mathcal K}] \le (1-p)^{{|Y|\choose k}}\le e^{-p{|Y|\choose k}}=e^{-p{t\choose k}}\le e^{-p(t/k)^k}.$$
 
 If $3\le k\le 3s^4$ then we can assume that $s=3$ in the condition in the Theorem (the statement is the strongest for $s=3$, provided the bound on $k$ is valid, which is the case here). There are sufficiently large constants $C',C'', C$ depending on $k,\epsilon$ such that  we have $$p(t/k)^k\ge C'n^{-\frac k{s+1}}t^k \ge C''n^{-\frac k{4}}\Big(\frac n {u^{1/3}}\Big)^k = Cn^{-\frac k{4}}\frac{n^{\frac {2k}{3}}}{\log^{k/3}n} 
 \ge n^{0.4k}\ge n^{1.2},$$ and so $${\rm P}[A_{\mathcal K}] \le e^{-n^{1.2}}.$$
 
Slightly modifying the calculations above, we see that for sufficiently small $\epsilon$ and $s = 4$ we can get $p(t/k)^k\ge n^{1.05}$ for $k=2.$

Recall that $s, \epsilon$ are constants. If $3s^4< k\le n^{\frac 12-\frac 1s}$ then we have 
\begin{multline*}
   p(t/k)^k\ge n^{-\frac k{s+1}}\Big(\frac{\epsilon n}{k^2u^{1/s}}\Big)^k = n^{-\frac k{s+1}}n^{\frac ks} \Big(\frac{\epsilon s^2}{\log n}\Big)^k\ge\\ 2^{\frac{k}{s(s+1)}\log n -k\log \frac 1\epsilon-k\log\log n}= 2^{(1-o(1))\frac{k}{s(s+1)}\log n}\ge n^{k/s^3}, 
\end{multline*}

 provided $n$ is sufficiently large (recall that $s\ge 3$). and so $${\rm P}[A_{\mathcal K}] \le e^{-n^{k/s^3}}.$$
 
 We are now ready to conclude the proof. Denote by $U$ the event that there exists a proper coloring of $KG_{n,k}$ that respects the lists chosen as described in the beginning of the proof.  
 \begin{align*}
     {\rm P}[U]\le \sum_{\mathcal K\in \mathcal C^u }{\rm P}[A_{\mathcal K}] \le n^{sk^su}\cdot \begin{cases}e^{-n^{1.05}},\ k=2\\ e^{-n^{1.2}},\ \ 3\le k\le 3s^4\\ e^{-n^{k/s^3}},\ k> 3s^4\end{cases} <1,
 \end{align*}
 provided $n$ is sufficiently large. Let us explain why does the inequality is valid in the third case. It is equivalent to showing that $sk^s u\log n<n^{k/s^3}.$ Note that $su\log n= o(n^2)$, and thus it is enough to show that $k^s<n^{k/s^3-2},$ or, equivalently, $s\log k< \big(\frac k{s^3}-2\big)\log n.$ Note that in the assumption $k>3s^3$ we have $s< \frac{k}{s^3}-2$. Since we have $\log k<\log n,$ the aforementioned inequality (and thus the displayed inequality) is valid. Thus, there exists a choice of lists of size $u/2$ each so that no proper coloring with such list is possible. This completes the proof of the theorem.\end{proof}
 \vskip+0.2cm
 {\sc Acknowledgements:} We thank the anonymous referees for carefully reading the text and pointing out several problems with the presentation. The research of the authors is  supported by the grant RSF N 21-71-10092,
\url{https://rscf.ru/project/21-71-10092/}

 \end{document}